\documentclass[11pt]{article}

\pdfoutput=1

\makeatletter
\let\@fnsymbol\@arabic
\makeatother

\makeatletter
\def\blfootnote{\gdef\@thefnmark{}\@footnotetext}
\makeatother

\usepackage{amsthm}
\usepackage{amsmath}
\usepackage{amssymb}
\usepackage{amsfonts}
\usepackage{svg}
\usepackage{graphicx}
\usepackage[ruled]{algorithm}
\usepackage[noend]{algpseudocode}
\usepackage{enumerate}
\usepackage{todonotes}
\usepackage{mathtools}
\usepackage{subcaption}
\mathtoolsset{showonlyrefs}
\usepackage{comment}
\usepackage{tikz-cd}
\usepackage{caption}
\usepackage{soul}
\usepackage[makeroom]{cancel}
\usepackage{setspace}
\usepackage{wrapfig}

\definecolor{darkblue}{rgb}{0.0, 0.0, 0.8}

\usepackage[colorlinks=true, allcolors=darkblue]{hyperref}

\hypersetup{
    pdftitle={Brownian motion on spaces of discrete regular curves},
    pdfauthor={Karen Habermann, Emmanuel Hartman},
}

\usepackage{array}
\newcommand{\PreserveBackslash}[1]{\let\temp=\\#1\let\\=\temp}
\newcolumntype{C}[1]{>{\PreserveBackslash\centering}p{#1}}
\newcolumntype{R}[1]{>{\PreserveBackslash\raggedleft}p{#1}}
\newcolumntype{L}[1]{>{\PreserveBackslash\raggedright}p{#1}}

\newtheorem{thm}{Theorem}
\newtheorem{lemma}[thm]{Lemma}
\newtheorem{proposition}[thm]{Proposition}

\newcommand{\R}{\mathbb{R}}
\newcommand{\Z}{\mathbb{Z}}

\newcommand{\vol}{\operatorname{vol}}

\newcommand{\dd}{\,{\mathrm d}}
\newcommand{\db}{{\mathrm d}}
\newcommand{\Db}{{\mathrm D}}

\usepackage[margin=1in]{geometry}

\title{Brownian motion on spaces of discrete regular curves\blfootnote{{\it $2020$ Mathematics Subject Classification.} 58J65, 60J65, 62R30}\blfootnote{{\it Key words and phrases.} shape analysis, Sobolev metrics, Brownian motion, stochastic completeness, numerics}}
\author{Karen Habermann\,\footnote{Department of Statistics, University of Warwick, Coventry, CV4 7AL, United Kingdom.\\ {\it Email address:} {\tt karen.habermann@warwick.ac.uk}} , Emmanuel Hartman\,\footnote{Department of Mathematics, University of Houston, Houston, TX, USA\\ {\it Email address:} {\tt ehartma2@CougarNet.UH.edu}}}
\begin{document}
\maketitle
\begin{abstract}
We introduce and study Brownian motion on spaces of discrete regular curves in Euclidean space equipped with discrete Sobolev-type metrics. It has been established that these spaces of discrete regular curves are geodesically complete if and only if the Sobolev-type metric is of order two or higher. By relying on a general result by Grigor'yan and controlling the volume growth of geodesic balls, we show that all spaces of discrete regular curves that are geodesically complete are also stochastically complete, that is, the associated Brownian motion exists for all times. This provides a rigorous footing for performing data statistics, such as data inference and data imputation, on these spaces. Our result is the first stochastic completeness result in shape analysis that applies to the full shape space of interest. For illustrative purposes, we include simulations for sample paths of Brownian motion on spaces of discrete regular curves. For the space of triangles in the plane modulo rotation, translation and scaling, we further provide heuristics which suggest that this space remains stochastically complete even for Sobolev-type metrics of order zero and one.
\end{abstract}

\section{Introduction}
In shape analysis, we frequently consider spaces of parameterized shapes as infinite-dimensional Riemannian manifolds and study the geometry of these spaces, see e.g.~\cite{BMM23,bruveris2015completeness,bruveris2014geodesic,MM06,michor2007overview}. However, in practice, we must discretize these infinite-dimensional spaces through various schemes, resulting in finite-dimensional analogs of the considered spaces which are endowed with Riemannian metrics that emulate the metric on the infinite-dimensional space.

In recent works, such as \cite{beutler2025discretegeodesiccalculusspace,cerqueira2024sobolevmetricsspacesdiscrete} some effort has been devoted to studying the geometry of these finite-dimensional manifolds to show consistency with known results in the infinite-dimensional setting as well as to provide insights into properties of the finite-dimensional space which are unknown in the infinite-dimensional setting.

In the present work, we obtain a stochastic completeness result for a discrete version of the space of closed immersed curves in $\R^d$ equipped with the reparametrization invariant Sobolev metric of order $m\in\Z_{\geq 0}$. This ensures long-time existence of Brownian motion on these spaces. It was raised in~\cite{cerqueira2024sobolevmetricsspacesdiscrete} as a problem of interest. The only previously known stochastic completeness result in shape analysis concerns the configuration space of exactly two landmarks in $\R^d$ and was established in~\cite{habermann2024long}.

\medskip

The space $\operatorname{Imm}(S^1,\R^d)$ of closed immersed curves in $\R^d$ for $d\geq 2$ is an infinite-dimensional Fr\'{e}chet manifold that can be equipped with the reparametrization invariant Sobolev metric $G^m$ of order $m\in\Z_{\geq 0}$. For $c\in \operatorname{Imm}(S^1,\R^d)$ and $h,k\in T_c\operatorname{Imm}(S^1,\R^d)\cong C^\infty(S^1,\R^d)$, this metric is defined in terms of arc length integration by
\begin{displaymath}
    G_c^m(h,k)= \int_{S^1} \left(\frac{\langle h,k \rangle}{l(c)^3}+\frac{\langle \Db_s^m h,\Db_s^m k\rangle}{l(c)^{3-2m}} \right)\db s,
\end{displaymath}
where $\langle\cdot,\cdot\rangle$ denotes the Euclidean inner product on $\R^d$, $l(c)$ denotes the length of the curve $c$ and $\Db_s^m$ denotes the $m$th-order derivative with respect to arc length. For further details, see e.g.~\cite{bruveris2015completeness,bruveris2014geodesic,cerqueira2024sobolevmetricsspacesdiscrete,michor2007overview}.

As a finite-dimensional analog of the space of closed immersed curves, we consider the space of closed piecewise linear curves with $n\geq 3$ breakpoints, which is exactly determined by the space $\R^{d\times n}$ of $n$ ordered points in $\R^d$, but where we disallow two adjacent vertices from coinciding. The latter emulates the immersion property of $\operatorname{Imm}(S^1,\R^d)$, resulting in the discrete space
\begin{displaymath}
    \R_*^{d\times n}= \left\{(v_0,\dots,v_{n-1}) \in \R^{d\times n} : v_i\neq v_{i+1} \text{ for all }i\in \Z/n\Z\right\}.
\end{displaymath}
Since this is an open subset of $\R^{d\times n}$, we have $T_v\R_*^{d\times n} = \R^{d\times n}$ for $v\in \R_*^{d\times n}$.
We denote the edges of a piecewise linear curve determined by $v\in \R_*^{d\times n}$ as $e_i(v) = v_{i+1}-v_i$ for all $i\in \Z/n\Z$. The total length of the piecewise linear curve $v$ is then given by $l(v)=\sum_{i=0}^{n-1}|e_i(v)|$, where $|\cdot|$ denotes the Euclidean norm on $\R^d$. We further define a discretization of the $m$th-order arc length derivative of a tangent vector $h\in T_v\R_*^{d\times n}$ by first setting $\Db_s^0 h = h$ and then recursively through, for all $i\in \Z/n\Z$,
\begin{displaymath}
    \Db_s^m h_i = \begin{cases} \dfrac{\Db_s^{m-1}h_{i+1} - \Db_s^{m-1}h_{i}}{|e_i(v)|}& \text{if } m \text{ is odd,}\\[1em] \dfrac{\Db_s^{m-1}h_{i} - \Db_s^{m-1}h_{i-1}}{\frac{1}{2}\left(|e_i(v)|+|e_{i-1}(v)|\right)} & \text{if } m \text{ is even.} \end{cases}
\end{displaymath}
With this at hand, we can define the discrete version $g^m$, introduced in~\cite{cerqueira2024sobolevmetricsspacesdiscrete}, of the reparametrization invariant Sobolev metric $G^m$ of order $m\in\Z_{\geq 0}$. Given $v\in \R_*^{d\times n}$ and $h,k\in  T_v\R_*^{d\times n}$, we set
\begin{displaymath}
    g_v^m(h,k)=\sum_{i=0}^{n-1}\left(\frac{\langle h_i,k_i\rangle}{l(v)^3} \cdot \frac{|e_i(v)|+|e_{i-1}(v)|}{2} + \frac{\langle \Db_s^m h_i,\Db_s^m k_i\rangle}{l(v)^{3-2m}}\cdot\mu_i\right),
\end{displaymath}
where
\begin{displaymath}
    \mu_i =\begin{cases}
        |e_i(v)|& \text{if } m \text{ is odd,} \\[1em]
        \dfrac{|e_i(v)|+|e_{i-1}(v)|}{2}& \text{if } m \text{ is even.}
    \end{cases}
\end{displaymath}
It was shown in~\cite{cerqueira2024sobolevmetricsspacesdiscrete} that $g^m$ indeed defines a Riemannian metric on $\R_*^{d\times n}$ for $m\in\Z_{\geq 0}$ and $n\geq 3$, see~\cite[Lemma 3.4]{cerqueira2024sobolevmetricsspacesdiscrete}, and that this discretization recovers the reparametrization invariant Sobolev metric $G^m$ on $\operatorname{Imm}(S^1,\R^d)$ in the limit $n\to\infty$, see~\cite[Proposition 3.7]{cerqueira2024sobolevmetricsspacesdiscrete}. Moreover, the work~\cite{cerqueira2024sobolevmetricsspacesdiscrete} studied completeness properties of the space $\R_*^{d\times n}$ equipped with $g^m$ for $m\in\Z_{\geq 0}$ and established the following result, which is part of~\cite[Theorem 4.1]{cerqueira2024sobolevmetricsspacesdiscrete}.
\begin{thm}[Cerqueira--Hartman--Klassen--Bauer~\cite{cerqueira2024sobolevmetricsspacesdiscrete}]\label{thm:geocomplete}
    Suppose $d\geq 2$ and $n\geq 3$. Then the space $\R_*^{d\times n}$ equipped with $g^m$ is geodesically complete if and only if $m\geq 2$.
\end{thm}

Let $\Delta^{d\times n,m}$ denote the Laplace--Beltrami operator of the Riemannian manifold $(\R_*^{d\times n},g^m)$ for $d\geq 2$, $n\geq 3$ and $m\in\Z_{\geq 0}$. Brownian motion on the space $\R_*^{d\times n}$ of discrete regular curves equipped with $g^m$ is then defined to be the stochastic process on $\R_*^{d\times n}$ with generator $\frac{1}{2}\Delta^{d\times n,m}$. Numerical simulations for Brownian motion on spaces of discrete regular curves are included in Section~\ref{sec:simulations}. The Riemannian manifold $(\R_*^{d\times n},g^m)$ is called stochastically complete if, with probability one, the induced Brownian motion exists for all times. In this work, we derive the following broad characterization.
\begin{thm}\label{thm:stocomplete}
    Suppose $d\geq 2$ and $n\geq 3$. Then the space $\R_*^{d\times n}$ equipped with $g^m$ is stochastically complete for all $m\geq 2$.
\end{thm}

We prove Theorem~\ref{thm:stocomplete} by showing that whenever $\R_*^{d\times n}$ equipped with $g^m$ is geodesically complete then it is also stochastically complete. This is not necessarily guaranteed in general, and the recent work~\cite{HPS} constructs a landmark configuration space that is geodesically complete but stochastically incomplete by using a kernel not arising from a Sobolev operator. We also highlight that the geodesic incompleteness of the space $\R_*^{d\times n}$ equipped with $g^m$ for $m\in\{0,1\}$ need not imply its stochastic incompleteness, with the punctured Euclidean plane being a prominent example of a space that is geodesically incomplete yet stochastically complete.

Although the notions of geodesic completeness and stochastic completeness are not linked straightforwardly, under additional control of the curvature or the volume growth of geodesic balls, it is known that geodesic completeness implies stochastic completeness. Specifically, we make use of the following characterization by Grigor'yan, see~\cite[Theorem~1]{grigoryan1} and \cite[Theorem~9.1]{grigoryan2}.
\begin{thm}[Grigor'yan~\cite{grigoryan1,grigoryan2}]\label{thm:grigoryan}
    Let $(M,g)$ be a finite-dimensional Riemannian manifold that is geodesically complete. If, for some $x\in M$, the volumes $V(r)$ of the geodesic balls with center $x$ and of radius $r>0$ satisfy, for some $a>0$,
    \begin{displaymath}
        \int_a^\infty\frac{r\dd r}{\log V(r)} =\infty,
    \end{displaymath}
    then $(M,g)$ is stochastically complete.
\end{thm}

In addition to establishing Theorem~\ref{thm:stocomplete}, we perform numerical simulations of Brownian motion on the Riemannian manifold $(\mathbb{R}_*^{d\times n}, g^m)$ using the Euler--Maruyama discretization of the associated stochastic differential equation. These simulations allow us to visualize stochastic evolution in the space of parameterized curves and to compare the qualitative behavior induced by Sobolev-type metrics of different orders.  

What remains to be studied in future work is the stochastic completeness property of $\R_*^{d\times n}$ equipped with $g^m$ for $m\in\{0,1\}$. To provide some indication that these spaces might well remain stochastically complete even for $m\in\{0,1\}$, we include a heuristic discussion for the space of triangles modulo rotation, translation and scaling. For this toy model, we can gain insight into the geometry near singularities and, thereby, into both geodesic and stochastic completeness properties of the space.

\paragraph{The paper is organized as follows.}
In Section~\ref{sec:stocomplete}, we first derive a control for the volume growth of geodesic balls in the space of discrete regular curves equipped with discrete Sobolev-type metrics of
order two or higher and then deduce Theorem~\ref{thm:stocomplete}, which guarantees long-time existence of the associated Brownian motion. Simulations for sample paths of Brownian motion on spaces of discrete regular curves are included in Section~\ref{sec:simulations}. We close with a heuristic analysis in Section~\ref{sec:triangles} of geodesic completeness and stochastic completeness for the space of triangles in the plane modulo rotation, translation and scaling.

\paragraph{Acknowledgements.}
The authors would like to thank the organizers and participants of the Summer Shape Workshops (\href{https://sites.google.com/view/shape-workshop/home}{https://sites.google.com/view/shape-workshop/home}) ``Math in Maine'' in 2024 and ``Math in Umbria'' in 2025. The authors would also like to thank the organizers of the thematic program ``Infinite-dimensional Geometry: Theory and Applications'' at the Erwin Schrödinger International Institute for Mathematics and Physics in Vienna and acknowledge the excellent working conditions at the institute. Most of this work was done at these events.

\section{Stochastic completeness}\label{sec:stocomplete}

We derive a control for the volume growth of geodesic balls in $\R_*^{d\times n}$ equipped with $g^m$ for $m\geq 2$, which then together with the criterion by Grigor'yan given in Theorem~\ref{thm:grigoryan} allows us to deduce our stochastic completeness result stated in Theorem~\ref{thm:stocomplete}.

\begin{lemma}\label{lem:rational_functions}
    Let $h\in\R^{d\times n}$ be arbitrary. Then, for all $v\in\R^{d\times n}_*$ and $m\in\Z_{\geq 0}$, we have that $g^m_v(h,h)$ is a rational function of the edge lengths $|e_i(v)|$ for $i\in \Z/n\Z$ and the denominator of its reduced form only has factors of the form $|e_i(v)|$, $|e_i(v)|+|e_{i-1}(v)|$ and $\sum_{i=0}^{n-1}|e_i(v)|$.
\end{lemma}
\begin{proof}
    We start by observing that any integer power of $l(v)=\sum_{i=0}^{n-1}|e_i(v)|$ is a rational function of the edge lengths $|e_i(v)|$ for $i\in \Z/n\Z$ with a denominator of the claimed form.
    In particular, it follows that
    \begin{displaymath}
        \sum_{i=0}^{n-1}\frac{\langle h_i,h_i\rangle}{l(v)^3} \cdot \frac{|e_i(v)|+|e_{i-1}(v)|}{2}
    \end{displaymath}
    is also a rational function of the edge lengths $|e_i(v)|$ for $i\in \Z/n\Z$ with a denominator satisfying the claimed property.
    Moreover, for the higher-order term, that is,
    \begin{displaymath}
        \sum_{i=0}^{n-1}\frac{\langle \Db_s^m h_i,\Db_s^m h_i\rangle}{l(v)^{3-2m}}\cdot\mu_i,
    \end{displaymath}
    it then suffices to establish that $\langle \Db_s^m h_i,\Db_s^m h_i\rangle$ for $i\in \Z/n\Z$ is likewise a rational function of the edge lengths with a denominator of the desired form. Formally, this can be shown via induction by first using that $\Db_s^0 h = h$ is a constant and then iterating through, for all $i\in \Z/n\Z$,
    \begin{displaymath}
        \Db_s^m h_i = \begin{cases} \dfrac{\Db_s^{m-1}h_{i+1} - \Db_s^{m-1}h_{i}}{|e_i(v)|}& \text{if } m \text{ is odd,}\\[1em] \dfrac{\Db_s^{m-1}h_{i} - \Db_s^{m-1}h_{i-1}}{\frac{1}{2}\left(|e_i(v)|+|e_{i-1}(v)|\right)} & \text{if } m \text{ is even,}\\ \end{cases}
    \end{displaymath}
    which componentwise takes the difference of rational functions whose denominators in reduced form only have factors of the form $|e_i(v)|$ as well as $|e_i(v)|+|e_{i-1}(v)|$ for $i\in \Z/n\Z$, and then divides by $|e_i(v)|$ or $|e_i(v)|+|e_{i-1}(v)|$, respectively, thereby retaining the desired property.
\end{proof}

Let $\vol_{g^m}$ denote the volume form on $\R^{d\times n}_*$ with respect to the Riemannian metric $g^m$. If $\mathbf{G^m}$ is the metric tensor of $g^m$ written in the standard Euclidean basis for $\R^{d\times n}$ and $\db V$ is the Euclidean volume form on $\R^{d\times n}$ then
\begin{displaymath}
    \vol_{g^m}=\sqrt{\det\mathbf{G^m}}\dd V.
\end{displaymath}
In particular, the following proposition, which is a consequence of Lemma~\ref{lem:rational_functions}, provides a control on the volume form $\vol_{g^m}$.

\begin{lemma}\label{lem:volume_control_via_edges}
    For all $m\in\Z_{\geq 0}$, the function $\det\mathbf{G^m}(v)$ is bounded above by a rational function of the edge lengths $|e_i(v)|$ for $i\in \Z/n\Z$ whose denominator in reduced form only has factors of the form $|e_i(v)|$, $|e_i(v)|+|e_{i-1}(v)|$ and $\sum_{i=0}^{n-1}|e_i(v)|$.
\end{lemma}
\begin{proof}
    Applying Hadamard's inequality for positive definite matrices to $\mathbf{G^m}=(\mathbf{G}_{k\ell}^\mathbf{m})_{1\leq k,\ell\leq dn}$, we obtain
    \begin{displaymath}
        \det\mathbf{G^m}(v)\leq \prod_{k=1}^{dn} \mathbf{G}_{kk}^\mathbf{m}(v).
    \end{displaymath}
    Since each matrix entry $\mathbf{G}_{kk}^\mathbf{m}(v)$ for $k\in\{1,\dots,dn\}$ can be expressed as $g^m_v(h,h)$, by letting $h$ run over the standard Euclidean basis for $\R^{d\times n}$, the claimed result then follows from Lemma \ref{lem:rational_functions}.
\end{proof}

In order to use Lemma~\ref{lem:volume_control_via_edges}, which gives a control of the volume form $\vol_{g^m}$ on $\R^{d\times n}_*$ in terms of the associated edge lengths, to analyze the volume growth of geodesic balls in the space $\R^{d\times n}_*$ equipped with $g^m$, we need to know that the edge lengths cannot increase or decrease too rapidly in growing geodesic balls. We let $B_{g^m}(v_0,r)$ denote the geodesic ball in $\R^{d\times n}_*$, with respect to the metric $g^m$, with center $v_0\in \R^{d\times n}_*$ and of radius $r>0$.

\begin{lemma}\label{lem:edge_growth}
    Let $m\geq 2$ and fix $v_0\in \R^{d\times n}_*$. There then exist constants $C_0,C_1>0$ such that, for all $r>0$, all $v\in B_{g^m}(v_0,r)$ and all $i\in \Z/n\Z$,
    \begin{displaymath}
        C_0\exp\left(-\frac{r}{2^{m-1}}\right)
        \leq |e_i(v)|
        \leq C_1\exp\left(\frac{r}{2^{m-1}}\right).
    \end{displaymath}
\end{lemma}
\begin{proof}
    For $v_0\in \R^{d\times n}_*$ fixed, we define the positive constants
    \begin{displaymath}
        C_0=\min_{i\in \Z/n\Z}|e_i(v_0)|
        \quad\text{and}\quad
        C_1 = \max_{i\in \Z/n\Z}|e_i(v_0)|.
    \end{displaymath}
    In~\cite[Proof of Lemma~4.4]{cerqueira2024sobolevmetricsspacesdiscrete}, it was derived that, for $m\geq 2$ and $v\colon [0,T)\to M$ with $\left.\frac{\db}{\db t}v(t)\right|_{t=0}=h$,
    \begin{equation}\label{eq:keyestimate}
        \left|\left.\frac{\db}{\db t}\log|e_i(v(t))|\right|_{t=0}\right|
        \leq\frac{1}{2^{m-1}}g_v^m(h,h).
    \end{equation}
    By integrating~\eqref{eq:keyestimate} along suitable geodesics, it follows that, for all $v\in B_{g^m}(v_0,r)$ and all $i\in \Z/n\Z$,
    \begin{displaymath}
        \left|\log\left(\frac{|e_i(v)|}{|e_i(v_0)|}\right)\right|
        =\left|\log|e_i(v)|-\log|e_i(v_0)|\right|
        \leq \frac{r}{2^{m-1}}.
    \end{displaymath}
    This in turn implies that
    \begin{displaymath}
        C_0\exp\left(-\frac{r}{2^{m-1}}\right)
        \leq |e_i(v_0)|\exp\left(-\frac{r}{2^{m-1}}\right)
        \leq |e_i(v)|
        \leq |e_i(v_0)|\exp\left(\frac{r}{2^{m-1}}\right)
        \leq C_1\exp\left(\frac{r}{2^{m-1}}\right),
    \end{displaymath}
    which establishes the claimed result.
\end{proof}

Combining Lemma~\ref{lem:volume_control_via_edges} and Lemma~\ref{lem:edge_growth}, we obtain the following control for the volume $V_{g^m}(v_0,r)$ of the geodesic ball $B_{g^m}(v_0,r)$ provided $m\geq 2$. Note that the assumption $m\geq 2$ is a consequence of the derivation provided for~\eqref{eq:keyestimate} in~\cite{cerqueira2024sobolevmetricsspacesdiscrete} requiring $m\geq 2$. The volume control we get is of the form one observes on Riemannian manifolds which admit a lower bound on the Ricci curvature by a negative constant.
\begin{proposition}\label{propn:growthcontrol}
    Fix $m\geq 2$ and $v_0\in \R^{d\times n}_*$. There then exist constants $D_0,D_1>0$ such that, for all $r>0$,
    \begin{displaymath}
        V_{g^m}(v_0,r)\leq \left(2r\right)^{dn}D_0\exp\left(D_1 r\right).
    \end{displaymath}
\end{proposition}
\begin{proof}
    By Lemma~\ref{lem:volume_control_via_edges}, there exist polynomial functions $P_1$ and $P_2$ such that, for all $v\in\R^{d\times n}_*$,
    \begin{displaymath}
        \det\mathbf{G^m}(v)
        \leq\frac{P_1(|e_0(v)|,\dots,|e_{n-1}(v)|)}{P_2(|e_0(v)|,\dots,|e_{n-1}(v)|)},
    \end{displaymath}
    where the polynomial function $P_2$ only has factors of the form $|e_i(v)|$, $|e_i(v)|+|e_{i-1}(v)|$ and $\sum_{i=0}^{n-1}|e_i(v)|$. Since $\det\mathbf{G^m}(v)>0$ everywhere, we may further assume that, for all $v\in\R^{d\times n}_*$,
    \begin{displaymath}
        P_1(|e_0(v)|,\dots,|e_{n-1}(v)|)>0
        \quad\text{and}\quad
        P_2(|e_0(v)|,\dots,|e_{n-1}(v)|)>0.
    \end{displaymath}
    Due to the stated property of the polynomial function $P_2$, it follows thanks to the lower bound in Lemma~\ref{lem:edge_growth} that there exist constants $D_2,D_3>0$ such that, for all $r>0$ and all $v\in B_{g^m}(v_0,r)$,
    \begin{displaymath}
        P_2(|e_0(v)|,\dots,|e_{n-1}(v)|)\geq D_2\exp\left(-D_3 r\right).
    \end{displaymath}
    The triangle inequality and the upper bound from Lemma~\ref{lem:edge_growth} further show that there exist constants $D_4,D_5>0$ such that, for all $r>0$ and all $v\in B_{g^m}(v_0,r)$,
    \begin{displaymath}
        P_1(|e_0(v)|,\dots,|e_{n-1}(v)|)\leq D_4\exp\left(D_5 r\right).
    \end{displaymath}
    Putting all bounds together implies that there exist constants $D_0,D_1>0$ such that, for all $r>0$ and all $v\in B_{g^m}(v_0,r)$,
    \begin{displaymath}
        \sqrt{\det\mathbf{G^m}(v)}\leq D_0\exp\left(D_1 r\right).
    \end{displaymath}
    We conclude that, for all $r>0$,
    \begin{displaymath}
        V_{g^m}(v_0,r)
        =\int_{B_{g^m}(v_0,r)}\vol_{g^m}(v)
        =\int_{B_{g^m}(v_0,r)}\sqrt{\det\mathbf{G^m}(v)}\dd V
        \leq \left(2r\right)^{dn}D_0\exp\left(D_1 r\right)
    \end{displaymath}
    to obtain the claimed result.
\end{proof}

Finally, we use the volume growth control for geodesic balls provided by Proposition~\ref{propn:growthcontrol} to prove Theorem~\ref{thm:stocomplete}.

\begin{proof}[Proof of Theorem~\ref{thm:stocomplete}]
    By Theorem~\ref{thm:geocomplete}, the space $\R_*^{d\times n}$ equipped with $g^m$ is geodesically complete for all $m\geq 2$. Moreover, by using Proposition~\ref{propn:growthcontrol}, we deduce that, for $v_0\in\R_*^{d\times n}$ and $a>0$,
    \begin{displaymath}
        \int_a^\infty\frac{r\dd r}{\log V_{g^m}(v_0,r)}
        \geq \int_a^\infty\frac{r\dd r}{D_1 r+dn\log\left(2r\right)+\log(D_0)}=\infty.
    \end{displaymath}
    The criterion by Grigor'yan stated in Theorem~\ref{thm:grigoryan} then implies that, for all $m\geq 2$, the space $\R_*^{d\times n}$ equipped with $g^m$ is stochastically complete.
\end{proof}

Our proof technique does not allow us to analyze the cases $m=0$ and $m=1$. It is also unclear whether the result~\cite[Theorem~4.2]{NN2} by Nenciu and Nenciu, which provides a criterion for a geodesically incomplete Riemannian manifold to be stochastically complete, could be applicable in this setting because the volume growth control obtained in Proposition~\ref{propn:growthcontrol} currently requires the assumption $m\geq 2$.

\section{Numerical simulations}\label{sec:simulations}
We perform numerical experiments that approximately simulate Brownian motion on the space $\mathbb{R}_*^{d\times n}$ of discrete regular curves equipped with the Riemannian metric $g^m$ for various $m\in\Z_{\geq 0}$. The Brownian paths are generated by using the Euler--Maruyama scheme to 
discretize the stochastic differential equation for the stochastic process with generator $\frac{1}{2}\Delta^{d\times n,m}$.

Starting from an initial configuration $v^0 \in \mathbb{R}_*^{d\times n}$, we want to construct a sequence $(v^k)_{k\ge0}$ according to, for a suitable diffusion matrix $\sigma$ and a drift term $b$,
\begin{displaymath}
    v^{k+1} = v^k + \Delta t\, b(v^k) + \sqrt{\Delta t}\,\sigma(v^k)\,\xi^k,
\end{displaymath}
where $\xi^k \sim \mathcal{N}(0,I_{d\times n})$ are independent standard Gaussian random variables. For this to correspond to Brownian motion on $(\mathbb{R}_*^{d\times n},g^m)$, the diffusivity $\sigma$ determined by the Riemannian metric $g^m$ needs to satisfy $\sigma\sigma^\top=(\mathbf{G}^m)^{-1}$, which can be computed via the Cholesky factorization of the inverse metric tensor $(\mathbf{G}^m)^{-1}$. Moreover, the drift term $b$ which arises from the geometry of the Riemannian manifold $(\mathbb{R}_*^{d\times n},g^m)$ can be written, in terms of index notation, as
\begin{displaymath}
    b^j=\frac{1}{2}\frac{1}{\sqrt{\det g^m}}\frac{\partial}{\partial x^i}
    \left(\sqrt{\det g^m} (g^m)^{ij}\right)= \frac{1}{2}\frac{\partial}{\partial x^i}(g^m)^{ij}+\frac{1}{2}g^{ij}\frac{\partial}{\partial x^i}\log\sqrt{\det g^m}.
\end{displaymath}
Hence, we can approximate $b$ efficiently by computing $\nabla \log\det(\mathbf{G}^m)$ and $\nabla (\mathbf{G}^m)^{-1}$ via automatic differentiation. We present pseudo-code for the discussed approach in Algorithm \ref{alg:BM}.

\begin{algorithm}[H]
\caption{Simulation of Brownian Motion on $(\mathbb{R}_*^{d\times n}, g^m)$ via Euler--Maruyama Scheme}\label{alg:BM}
\begin{algorithmic}[1]
\Require Initial configuration $v^0 \in \mathbb{R}_*^{d\times n}$, time step $\Delta t$, number of steps $N$
\For{$k = 0$ to $N-1$}
    \State Compute the metric tensor $\mathbf{G}^m(v^k)$
    \State Compute its inverse $(\mathbf{G}^m)^{-1}(v^k)$
    \State Obtain diffusion matrix $\sigma(v^k)$ via Cholesky factorization:
        \[
        \sigma(v^k)\sigma(v^k)^\top = (\mathbf{G}^m)^{-1}(v^k)
        \]
    \State Compute drift term $b(v^k)$\\
        \Comment{Efficiently computed via automatic differentiation of $\log\det(\mathbf{G}^m)$ and $(\mathbf{G}^m)^{-1}$}
    \State Sample $\xi^k \sim \mathcal{N}(0, I_{d\times n})$
    \State Update configuration:
        \[
        v^{k+1} = v^k + \Delta t\, b(v^k) + \sqrt{\Delta t}\, \sigma(v^k)\, \xi^k
        \]
\EndFor
\State \Return $\{v^k\}_{k=0}^{N}$
\end{algorithmic}
\end{algorithm}

\subsection{Simulations of Brownian motion}

We first visualize Brownian motion initialized at several different discrete regular curves. Figure~\ref{fig:brownian} shows three examples in the space of discrete regular curves with respect to the metric $g^2$. The evolution of the centroid of each curve is plotted for $t \in [0,10]$ with step size $\Delta t = 0.01$, along with the corresponding trajectory modulo translation.

\begin{figure}[h]
\centering
\includegraphics[width=\textwidth]{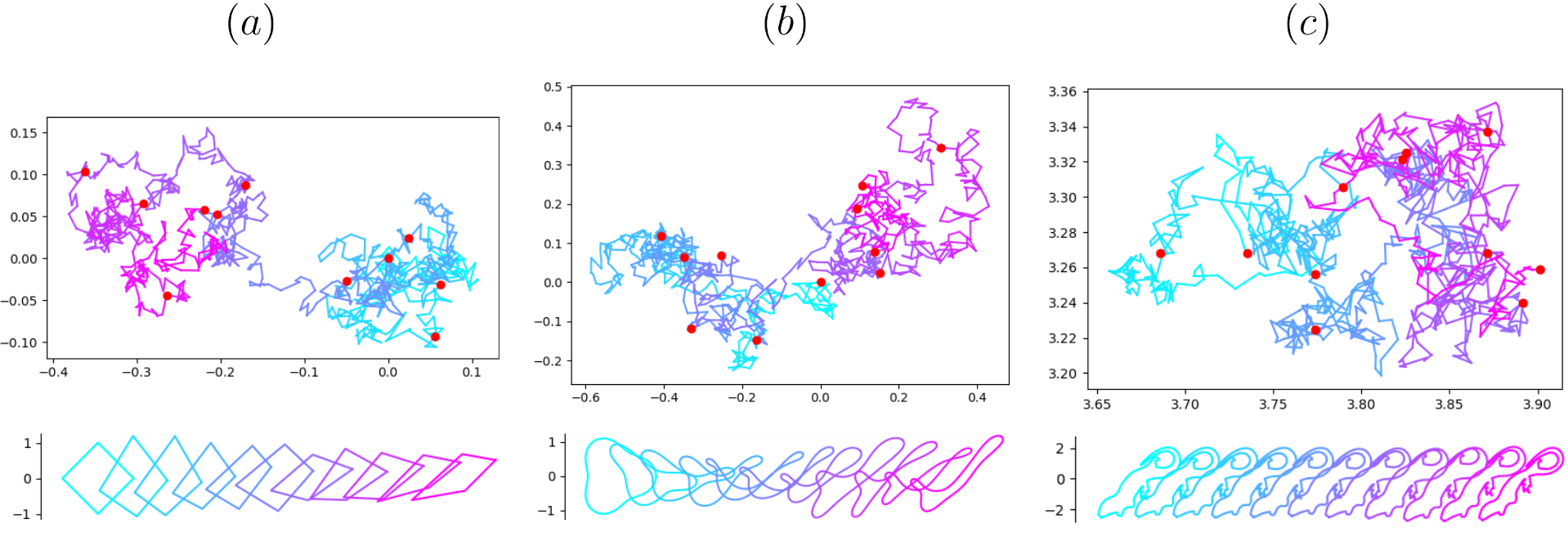}
\caption{Three examples of Brownian motion on the space of discrete regular curves with respect to the metric $g^2$. 
Top: Evolution of the centroid of each curve from $t=0$ to $t=10$ with step size $\Delta t =0.01$. 
Bottom: Corresponding trajectories of curves modulo translation. We plot every 10th curve corresponding to the centroid highlighted in red on the top plot.  
(a) Initialization at a square with $n=4$ points. 
(b) Initialization at a smooth curve with $n=100$. 
(c) Initialization at a shape from the MPEG7 dataset (lizard) with $n=100$.}
\label{fig:brownian}
\end{figure}

\subsection{Comparison of metric orders}

To further investigate the effect of the metric, we compare Brownian motion on a space of discrete regular curves in $\R^2$ for large $n$ and endowed with the 1st-, 2nd-, and 4th-order Sobolev-type metric, respectively. 
This experiment illustrates how the order of the metric influences regularity, diffusion rate and the preservation of geometric structure. 
In Figure~\ref{fig:orders}, we show qualitative examples of Brownian motion initialized at the unit circle in $\R^2$ with $n=100$ for the metrics $g^1$, $g^2$ and $g^4$, respectively.
In addition, we plot the minimum edge length of the evolving curve over time across multiple runs, providing empirical evidence supporting our theoretical results. Notably, even for the metric $g^1$, we have not observed collapse of edge lengths to zero over a time period from $t=0$ to $t=1000$. This is in line with our subsequent heuristic discussion regarding stochastic completeness for the space of triangles in the plane modulo rotation, translation and scaling.

\begin{figure}[ht]
\centering
\includegraphics[width=\textwidth]{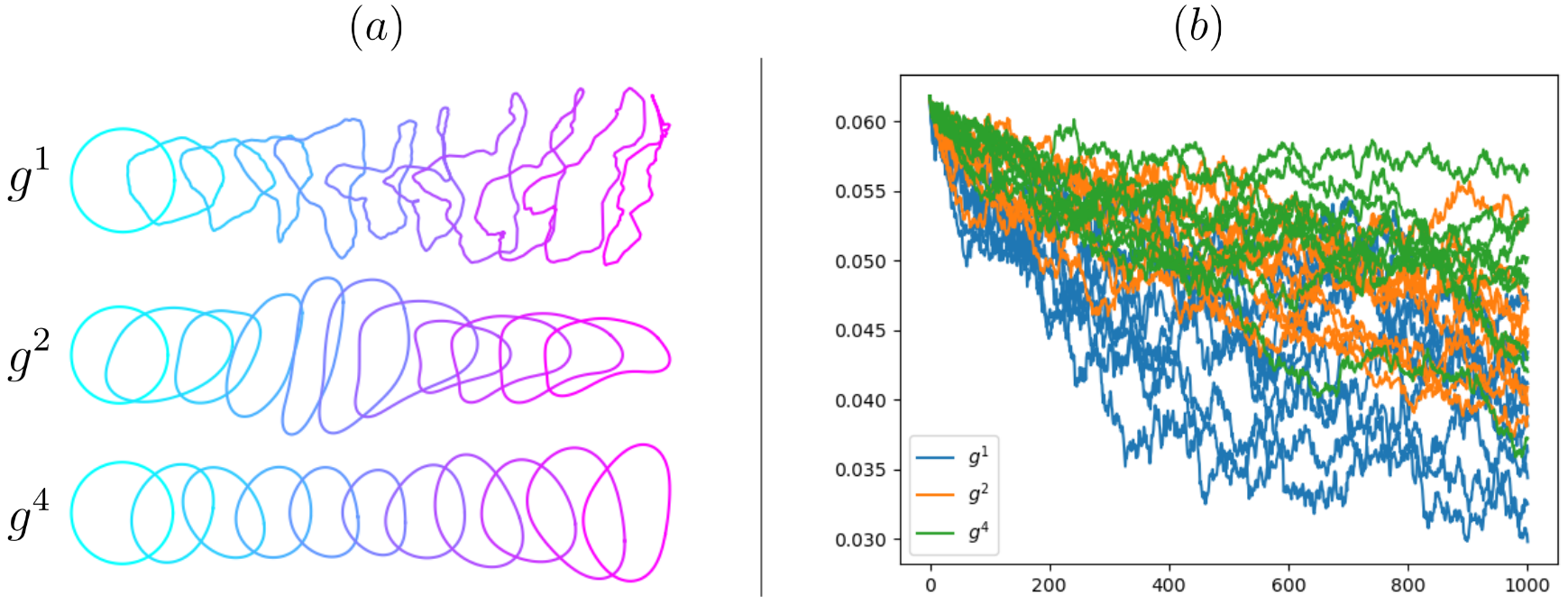}
\caption{(a) Brownian motion on the space of discrete regular curves in $\R^2$ for $n=100$ with respect to 1st-, 2nd-, and 4th-order Sobolev-type metrics, initialized at the unit circle discretized with $100$ points. 
(b) Evolution of the minimum edge length over time for $10$ independent simulations from $t=0$ to $t=1000$ each with respect to $g^1$, $g^2$ and $g^4$. }
\label{fig:orders}
\end{figure}

\section{Case study -- Triangles modulo rotation, translation and scaling}
\label{sec:triangles}

We provide a heuristic discussion of the stochastic completeness properties for the space of triangles modulo rotation, translation and scaling, which exploits a conformal structure underlying this setting and which particularly covers the cases $m\in\{0,1\}$.

As illustrated in the numerical simulations in the previous section, letting a discrete regular curve evolve according to Brownian motion on a space $\R_*^{d\times n}$ of discrete regular curves equipped with a Riemannian metric $g^m$ results both in a random movement of the curve and a random change of the edge lengths of the curve. If one is less interested in the random movement of the curve itself, it can be beneficial to consider discrete regular curves modulo rotation, translation and scaling. This constitutes a submanifold of the space $\R_*^{d\times n}$ that can be endowed with a Riemannian metric by restricting $g^m$ to the submanifold, which does not correspond to the inherited quotient metric.

The space of triangles modulo rotation, translation and scaling gives rise to a nice toy example, where one can gain insight into the geometry near singularities and, thereby, into both geodesic and stochastic completeness properties. We keep the discussion heuristic as it exploits a conformal structure which is not available in general. Moreover, while we expect our volume growth control for $m\geq 2$ to carry through to the spaces of discrete regular curves modulo rotation, translation and scaling, their geodesic completeness properties have not been studied formally in~\cite{cerqueira2024sobolevmetricsspacesdiscrete}.

When considering triangles in the plane modulo rotation, translation and scaling, we may take $v_0=(1,0)$ after translation, $|e_2(v)|=2$ after scaling, and $v_2=(-1,0)$ after rotation. This results in the discrete space $\overline{\R}^{2\times 3}_* = \R^2\setminus \{(1,0),(-1,0)\}$, with tangent spaces $T_v\overline{\R}^{2\times 3}_*=\R^2$ for $v\in\overline{\R}^{2\times 3}_*$. By identifying $h\in T_v\overline{\R}^{2\times 3}_*$ with $(0,h,0)\in T_v\R^{2\times 3}_*$, we can determine the restriction of the Riemannian metric $g^m$ to the submanifold $\overline{\R}^{2\times 3}_*$. It turns out to be of the form $f_m(v)\langle\cdot,\cdot\rangle$, that is, a metric conformal to the Euclidean metric on $\R^2\setminus \{(1,0),(-1,0)\}$ with conformal factor $f_m\colon\R^2\setminus \{(1,0),(-1,0)\}\to\R_{\geq 0}$. For $m\geq 1$, the conformal factor develops singularities at the points $(1,0)$ and $(-1,0)$, which corresponds to the vertex $v_1$ coinciding with $v_0$ and $v_2$, respectively. We obtain 
\begin{align*}
    f_0(v)&=\frac{|e_0(v)|+|e_1(v)|}{l(v)^3}, \\
    f_1(v)&=\frac{|e_0(v)|+|e_1(v)|}{2l(v)^3}+\left(\frac{1}{|e_0(v)|}+\frac{1}{|e_1(v)|}\right)\frac{1}{l(v)} \quad\text{and}\\
    f_2(v)&=\frac{|e_0(v)|+|e_1(v)|}{2l(v)^3}+\left(\frac{2}{|e_0(v)|^2(|e_0(v)|+2)}+\frac{2(|e_0(v)|+|e_1(v)|)}{|e_0(v)|^2|e_1(v)|^2}+\frac{2}{|e_1(v)|^2(|e_1(v)|+2)}\right)l(v).
\end{align*}

We now study how these conformal factors behave near $(1,0)$ and $(-1,0)$. By symmetry, it suffices to consider the case where the vertex $v_1$ approaches the vertex $v_0$, that is, where $|e_0(v)|=r$ becomes small. We first observe that $|e_1(v)|=2+o(1)$ as $r\downarrow 0$ and then further compute that, as $r\downarrow 0$, 
\begin{displaymath}
    f_0(v)=\frac{1+o(1)}{32},\quad
    f_1(v)=\frac{1+o(1)}{4r}\quad\text{and}\quad
    f_2(v)=\frac{8+o(1)}{r^2}.
\end{displaymath}
Thus, locally near $(1,0)$ and $(-1,0)$, respectively, the conformal factor corresponds to a perturbation of a scaled Euclidean metric for $m=0$, a standard cone with its tip at the singularity for $m=1$ and a cylinder for $m=2$. The conformal factors $f_0$, $f_1$ and $f_2$ are plotted in Figure~\ref{fig:conformalfactors}, together with an illustration of their behavior near the vertex $v_0=(1,0)$.

\begin{figure}[h]
\centering
\includegraphics[width=.49\textwidth]{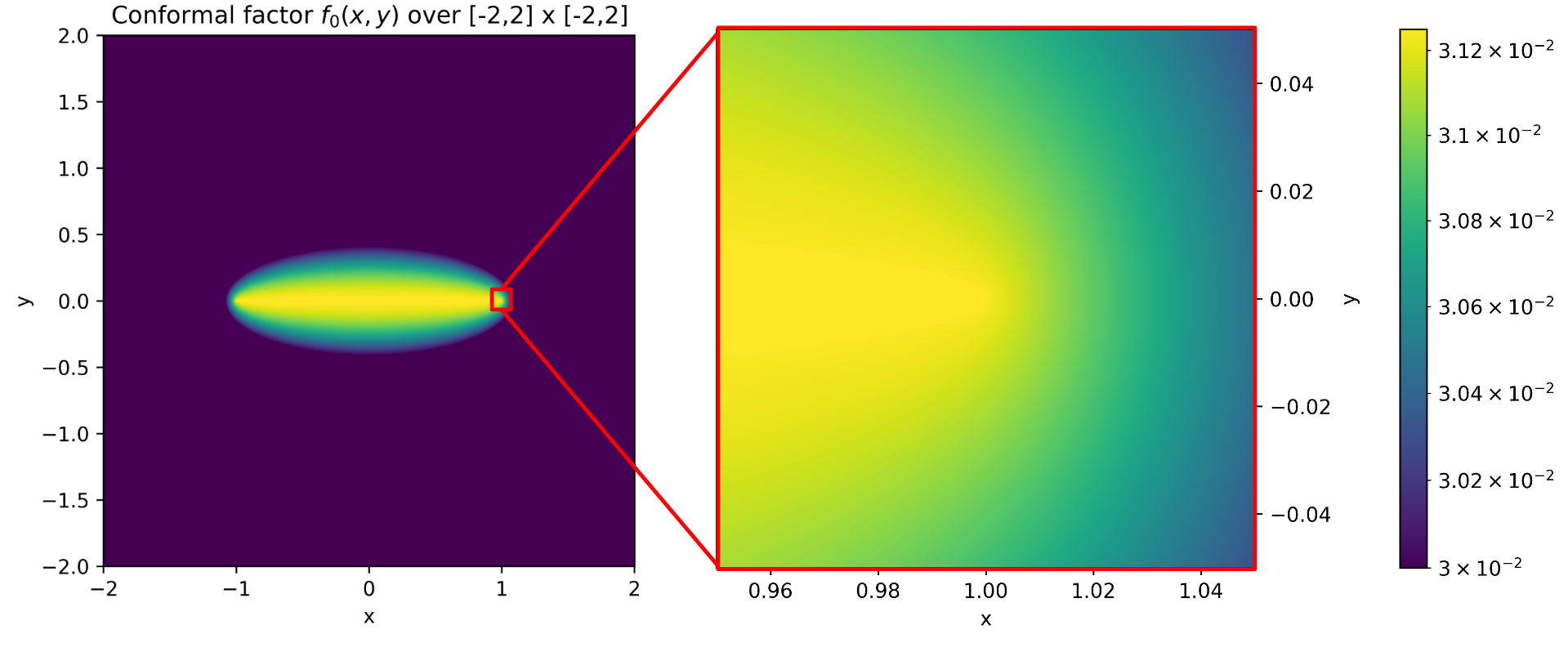}
\includegraphics[width=.49\textwidth]{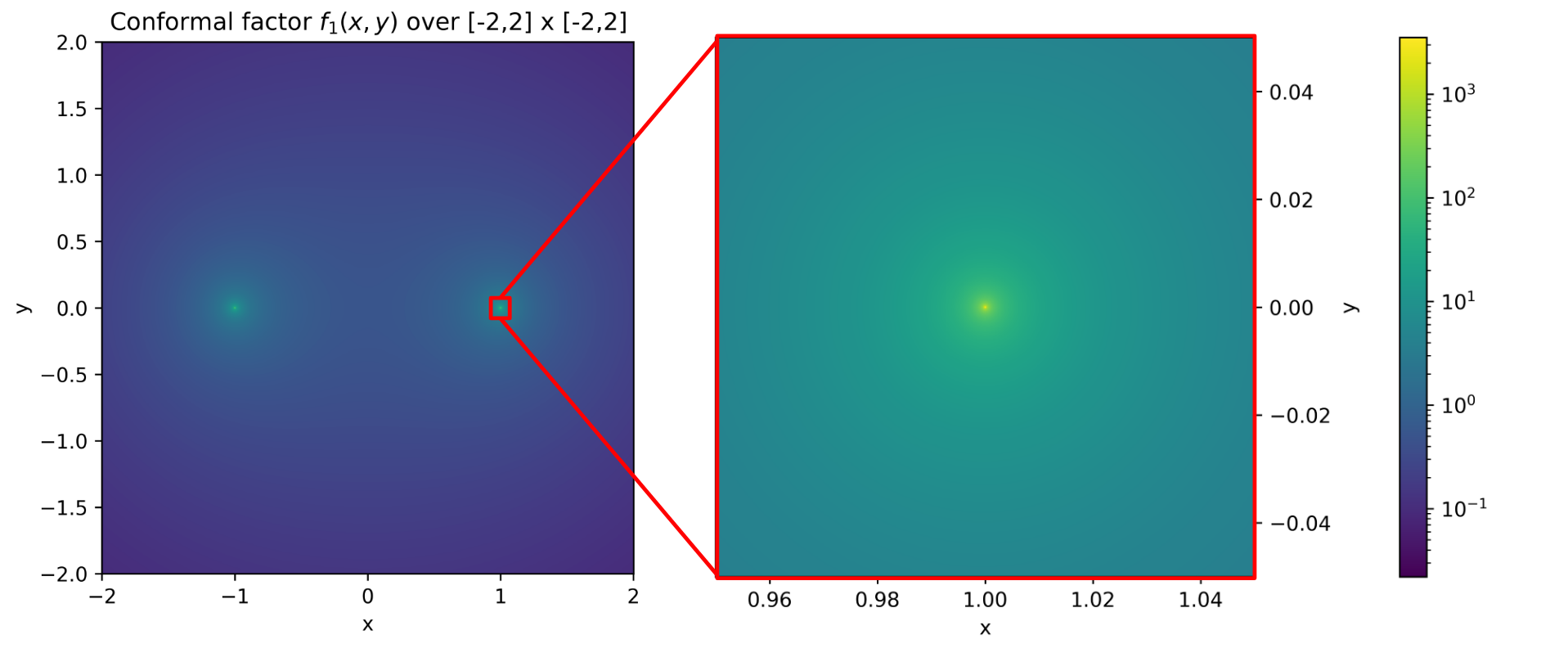}\\
\includegraphics[width=.49\textwidth]{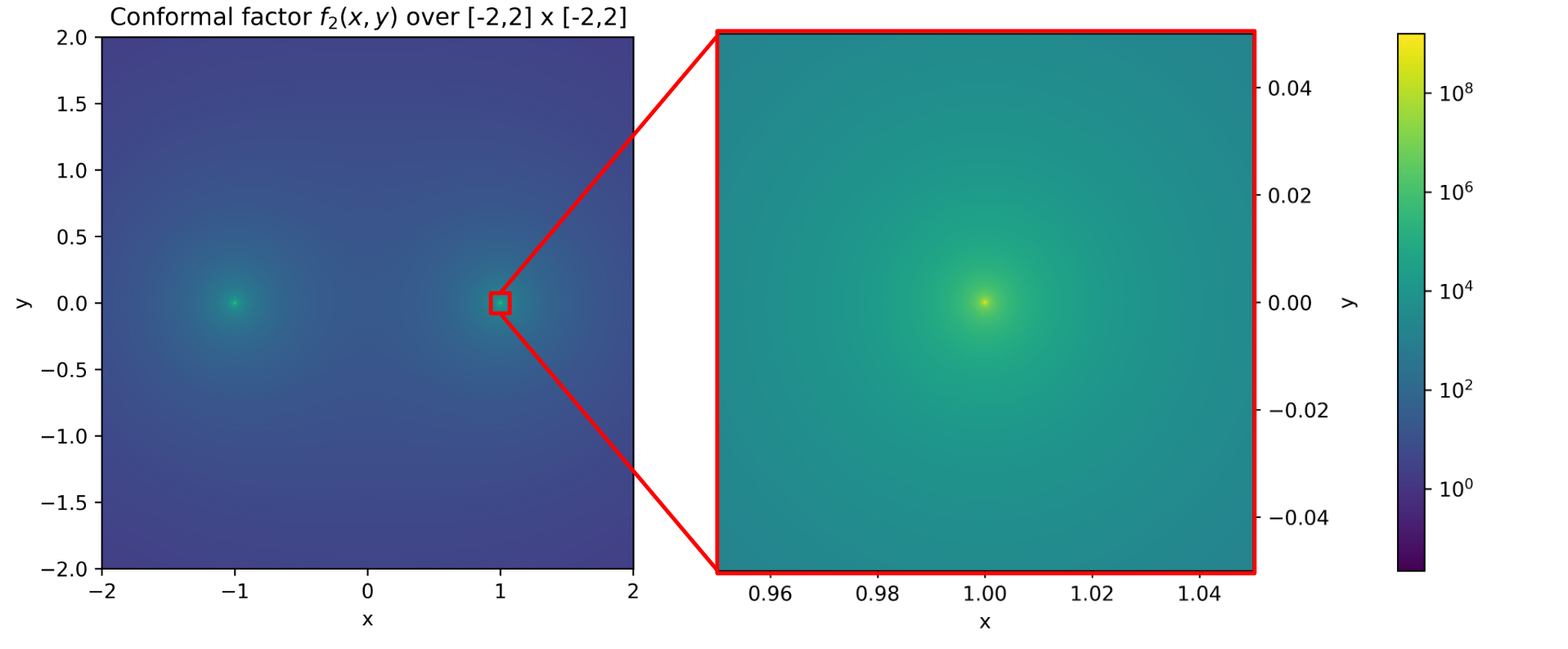}
\caption{Plots of the conformal factors $f_0$, $f_1$ and $f_2$ corresponding to the restrictions of the metrics $g^0$, $g^1$ and $g^2$ to the space of triangles with two fixed vertices. We display the conformal factors on $[-2,2]\times[-2,2]$ as well as a small region around the singularity at $v_0=(1,0)$.}
\label{fig:conformalfactors}
\end{figure}

As each iteration in the definition of the metric $g^m$ introduces one additional division by $|e_0(v)|$, we expect that, more generally, for all $m\in\Z_{\geq 0}$, there exists a constant $C_m>0$ such that, as $r\downarrow 0$,
\begin{displaymath}
    f_m(v)=\frac{C_m+o(1)}{r^m}.
\end{displaymath}
In particular, for $m\geq 3$, concentric circles around the singularity have increasing circumferences as their Euclidean radius decreases, and the conformal factor corresponds to a perturbation of a structure that opens up.

These observations that close to $(1,0)$ and $(-1,0)$ the space $\overline{\R}^{2\times 3}_*$ equipped with the restriction of $g^m$ essentially looks like a plane for $m=0$, a cone for $m=1$, a cylinder for $m=2$ and an opening structure for $m\geq 3$ is in line with the expectation that we have geodesic completeness if and only if $m\geq 2$ as well as stochastic completeness if $m\geq 2$. However, what is intriguing is that these correspondences also suggest that $\overline{\R}^{2\times 3}_*$ remains stochastically complete even for $m\in\{0,1\}$ because both the punctured plane and a cone with its tip removed are stochastically complete, despite them being geodesically incomplete. The work~\cite{HPS} provided a family of shape spaces that are geodesically complete yet stochastically incomplete. We now have with $(\overline{\R}^{2\times 3}_*,g^0)$ and $(\overline{\R}^{2\times 3}_*,g^1)$ first candidates for shape spaces that are geodesically incomplete yet stochastically complete. For a rigorous analysis, it will also be important to study possible explosion towards Euclidean infinity.

\bibliographystyle{abbrv}

\end{document}